\documentclass[12pt]{amsart}
\usepackage{amsmath,amssymb,amsthm}

\def\P{{\mathbb{P}}}

\def\p{{\mathfrak{p}}}

\def\C{\mathbb{C}}
\def\Z{\mathbb{Z}}
\def\Q{\mathbb{Q}}
\def\R{\mathbb{R}}
\def\F{\mathbb{F}}

\def\PGL{{\rm PGL}}

\def\im{{\rm im}\,}

\def\GL{{\rm GL}}

\def\Z{{\mathbb Z}}

\newcommand\Fp{\mathbb{F}_p}

\def\P{{\bf{P}}}

\newenvironment{Gal}{{\rm Gal\,}}{}

\newenvironment{Disc}{{\rm Disc \,}}{}
\newenvironment{Stab}{{\rm Stab}}{}
\newenvironment{Aut}{{\rm Aut}}{}

\newtheorem{theorem}{Theorem}[section]

\newtheorem{proposition-definition}[theorem]{Proposition-Definition}

\newtheorem{conjecture}[theorem]{Conjecture}

\theoremstyle{definition}

\newtheorem{question}[theorem]{Question} 

\theoremstyle{remark}

\begin{document}
\title{Galois representations from pre-image trees: an arboreal survey}
\author{Rafe Jones}

\begin{abstract}
Given a global field $K$ and a rational function $\phi \in K(x)$, one may take pre-images of $0$ under successive iterates of $\phi$, and thus obtain an infinite rooted tree $T_\infty$ by assigning edges according to the action of $\phi$. The absolute Galois group of $K$ acts on $T_\infty$ by tree automorphisms, giving a subgroup $G_\infty(\phi)$ of the group $\Aut(T_\infty)$ of all tree automorphisms. Beginning in the 1980s with work of Odoni, and developing especially over the past decade, a significant body of work has emerged on the size and structure of this Galois representation. These inquiries arose in part because knowledge of $G_\infty(\phi)$ allows one to prove density results on the set of primes of $K$ that divide at least one element of a given orbit of $\phi$. 

Following an overview of the history of the subject and  two of its fundamental questions, we survey in Section \ref{generic} cases where $G_\infty(\phi)$ is known to have finite index in $\Aut(T_\infty)$. While it is tempting to conjecture that such behavior should hold in general, we exhibit in Section \ref{exceptional} four classes of rational functions where it does not, illustrating the difficulties in formulating the proper conjecture. Fortunately, one can achieve the aforementioned density results with comparatively little information about $G_\infty(\phi)$, thanks in part to a surprising application of probability theory, as we discuss in Section \ref{density}. Underlying all of this analysis are results on the factorization into irreducibles of the numerators of iterates of $\phi$, which we survey briefly in Section \ref{stability}. We find that for each of these matters, the arithmetic of the forward orbits of the critical points of $\phi$ proves decisive, just as the topology of these orbits is decisive in complex dynamics. 
\end{abstract}

 \dedicatory{Dedicated to the late R. W. K. Odoni, whose inquisitive spirit led him before any others to these beautiful questions.} 

\maketitle

\section{Introduction}

In this survey, we lay out recent work on the action of the absolute Galois group of a global field
on trees of iterated pre-images under rational functions. These actions, also known as \textit{arboreal Galois representations}, have recently seen a surge in interest, largely due to their applications to certain density questions. Their study dates to the foundational work of R.~W.~K.~Odoni \cite{odonigalit, odoniwn, odoni} in the 1980s. Odoni aimed in part to study recurrence sequences satisfying relations of the type $a_n = f(a_{n-1})$, where $a_0 \in \Z$ and $f(x) \in \Z[x]$ is a polynomial of degree at least two. Such a sequence may be described as the orbit of $a_0$ under the dynamical system given by iteration of $f(x)$. 
One might ask whether the sequence $(a_n)_{n \geq 0}$ contains infinitely many primes, but this seems completely out of reach at present. Indeed, the sequence $(a_n)$ grows extremely quickly -- on the order of $d^{d^n}$ -- and a heuristic argument suggests that only finitely many of the $a_n$ are prime. To illustrate the difficulty of this problem, note that taking $a_0 = 3$ and $f(x) = (x-1)^2 + 1$ yields the Fermat numbers, whose prime decompositions have been a mystery since Fermat first speculated about them in 1640. 
A more reasonable hope is to obtain some qualitative information about the prime factorizations of the $a_n$, for instance by considering the whole collection 
$$P_f(a_0) := \{\text{$p$ prime} : \text{$p$ divides at least one non-zero term of $(a_n)_{n \geq 0}$}\}.$$ If this set is sparse within the set of all primes, then at least the $a_n$ do not in the aggregate have too many small prime factors. Another natural question, which we do not discuss in this survey, is whether all but finitely many terms of the sequence $(a_n)$ have a primitive prime divisor (that is, a prime divisor that does not divide any previous terms of the sequence). For a sampling of the large and interesting literature on this question, which merits a survey of its own, see \cite{faber, abc, ingram-silverman, krieger, rice, jhszig}.   

It was Odoni who in \cite{odonigalit, odoniwn} first recognized that if the Galois groups $G_n(f)$ of the iterates $f^n(x)$ of $f(x)$ satisfy certain properties, then $P_f(a_0)$ has natural density zero in the set of all primes (see p.~\pageref{densedef} for a definition of natural density). Indeed, the density of the complement of $P_f(a_0)$ is bounded below by the density of $p$ such that $f^n(x) \equiv 0 \bmod{p}$ has no solution (see p.~\pageref{bound} for more on this). \label{frobdisc} The latter condition is equivalent to Frobenius at $p$ acting without fixed points on the roots of $f^n(x)$. One then gets from the Chebotarev density theorem (in fact, the Frobenius density theorem suffices \cite[Section 3]{stevenhagen}) that $P_f(a_0)$ has density zero if 
\begin{equation} \label{thekey0}
\lim_{n \to \infty} \frac{\# \{g \in G_n(f) : \text{$g$ fixes at least one root of $f^n(x)$} \}}{\#G_n(f)} = 0.
\end{equation}


Odoni exploits this observation in \cite{odoniwn}, where he considers \textit{Sylvester's sequence}\footnote{Named for J. J. Sylvester, and known widely for its connections to Egyptian fractions.}, defined by
$$w_1 = 2, \qquad \text{$w_n = 1 + w_1w_2 \cdots w_{n-1}$ \; for $n \geq 2$}.$$
One readily checks that $w_n = w_{n-1}^2 - w_{n-1} + 1$, and so Sylvester's sequence is the orbit of $2$ under iteration of $f(x) = x^2 - x + 1$. Odoni proves the highly non-trivial result that $P_f(2)$ has density zero in the set of all primes by establishing isomorphisms
\begin{equation} \label{odonisum}
G_n(f) \cong \Aut(T_n) \qquad \text{for all $n \geq 1$},
\end{equation}
where $G_n(f)$ is the Galois group of the $n$th iterate of $f(x) = x^2 - x + 1$ and $\Aut(T_n)$ is the group of tree automorphisms of the complete binary rooted tree of height $n$. The tree in question has as its vertex set the disjoint union $\{0\} \sqcup f^{-1}(0) \sqcup f^{-2}(0) \sqcup \cdots \sqcup f^{-n}(0)$ of iterated preimages of $0$ under $f(x)$, and two vertices are joined by an edge when $f$ sends one vertex to the other. That $G_n(f)$ injects into $\Aut(T_n)$ follows from basic Galois theory; to prove surjectivity requires the art. With the explicit description of $G_n(f)$ given in \eqref{odonisum}, Odoni goes on to establish \eqref{thekey0} by a direct calculation \cite[p. 5]{odoniwn}, a result which has a nice restatement in terms of branching processes \cite[Proposition 5.5]{galmart}. It is worth pointing out that isomorphisms such as those in \eqref{odonisum} do not hold for $f(x) = (x-1)^2 + 1$; in this case $G_n(f)$ may be shown to be abelian, and the corresponding zero-density result follows easily \cite[p. 11]{odoniwn}.  

Odoni did not use the language of tree automorphisms, preferring to think of $\Aut(T_n)$ as the $n$-fold iterated wreath product of $\Z/2\Z$ (or more generally of $S_d$ when the tree is $d$-ary for $d \geq 2$). For us, considering elements of $G_n(f)$ as tree automorphisms has the advantage of providing an object on which Galois acts, thus allowing a more direct analogy with Galois representations associated to abelian varieties. We note that another dynamical Galois representation comes from the natural Galois action on the set of periodic points of $\phi$. We do not treat this interesting topic in the present article, but see \cite{morton}, \cite{pmp}, and \cite[Section 3.9]{jhsdynam}.

\subsection{Definitions and main questions}
To more closely match the Tate module from the theory of abelian varieties, we wish to attach an infinite pre-image tree to any rational function $\phi \in K(x)$ of degree $d \geq 2$ and any point $\alpha \in \mathbb{P}^1(K)$, where $K$ denotes a global field with separable closure $K^{\text{sep}}$. Denote by $\phi^n(x)$ the $n$th iterate of $\phi$, that is, the $n$-fold composition of $\phi$ with itself. We must be careful to consider \textit{only $\alpha$ for which the equation $\phi^{n}(x) = \alpha$ has $d^n$  distinct solutions, for each $n \geq 1$}. This ensures that we obtain a complete infinite rooted $d$-ary tree $T_\infty(\alpha)$ whose set of vertices is 
\begin{equation} \label{tdef}
\bigsqcup_{n \geq 0} \phi^{-n}(\alpha) \subseteq \mathbb{P}^1(K^{\text{sep}})
\end{equation}
and whose edges are given by the action of $\phi$ (we take $\phi^0(\alpha) = \{\alpha\}$ in \eqref{tdef}, and note that $\alpha$ is the root of the tree). The absolute Galois group $\Gal(K^{\text{sep}}/K)$ acts on $T_\infty(\alpha)$, and moreover preserves the connectivity relation in $T_\infty(\alpha)$, as Galois elements commute with $\phi$ since the latter is defined over $K$. Hence we obtain a homomorphism
\begin{equation*} \label{rhodef}
\rho : \Gal(K^{\text{sep}}/K) \to \Aut(T_\infty(\alpha)).
\end{equation*} 
The image of $\rho$ is the primary object of study in this article, and we write
\begin{equation*}
G_\infty(\phi, \alpha) := \im \rho.
\end{equation*}
More concretely, $G_\infty(\phi, \alpha)$ is the inverse limit of the Galois groups 
\begin{equation*}
G_n(\phi, \alpha) := \Gal(K(\phi^{-n}(\alpha))/K), 
\end{equation*}
which form an inverse system under the natural surjections $G_{n+1}(\phi, \alpha) \to G_n(\phi, \alpha)$ that arise from the inclusions $K(\phi^{-n}(\alpha)) \subseteq K(\phi^{-(n+1)}(\alpha))$. If $h$ is a M\"obius transformation defined over $K$ and $\psi := h^{-1} \circ \phi \circ h$, then a simple calculation shows that $K(\phi^{-n}(\alpha)) = K(\psi^{-n}(h^{-1}(\alpha)))$ for each $n \geq 1$. Taking $h$ to be translation by $\alpha$, we see that to determine $G_\infty(\phi, \alpha)$, we need only determine $G_\infty(\psi, 0)$, and hence to obtain complete knowledge in the general situation it is enough to understand the case where $\alpha = 0$. In the sequel, we thus drop any reference to $\alpha$ and write 
\begin{equation*}
\framebox{$T_\infty$ for $T_\infty(0)$, \qquad $G_\infty(\phi)$ for $G_\infty(\phi,0),$ \qquad
$G_n(\phi)$ for $G_n(\phi,0)$.}
\end{equation*}
In light of the definition of $\rho$, we have natural injections 
$$\text{$G_\infty(\phi) \hookrightarrow \Aut(T_\infty)$} \qquad \text{and} \qquad \text{$G_n(\phi) \hookrightarrow \Aut(T_n)$},$$ where the vertex set of $T_n$ is $\bigsqcup_{i = 0}^n \phi^{-i}(0)$, and edges are assigned according to the action of $\phi$.  
We emphasize that throughout this article, 
$$\text{we assume that for each $n \geq 1$, $\phi^n(x) = 0$ has $d^n$ distinct solutions in $K^{\text{sep}}$.}$$
This assumption is a mild one, and can be easily checked for a given $\phi$.  
With these conventions in place, we pose our first main question: 
\begin{question} \label{imagequestion} 
Let $K$ be a global field. 
\begin{enumerate}
\item[(a)] For which rational functions $\phi \in K(x)$ do we have $[\Aut(T_\infty) : G_\infty(\phi)] < \infty$? 
\item[(b)] For which $\phi$ do we have $G_\infty(\phi) = \Aut(T_\infty)$? 
\end{enumerate}
\end{question}
We remark that the finite index question is perhaps more robust, since a positive answer implies a positive answer when $K$ is replaced by any finite extension. 
In the well-studied case of $\ell$-adic Galois representations arising from elliptic curves, J.-P. Serre settled the analogue to Question \ref{imagequestion}(a) with his celebrated open image theorem \cite{serre1}. Let $E$ be an elliptic curve without complex multiplication and defined over a number field $K$, $\ell$ a rational prime, and $G_\infty$ the inverse limit of the Galois groups of the extensions $K(E[\ell^n])/K$. Because of the group structure on $E$, one has a natural injection $G_\infty \hookrightarrow \GL(2, \Z_\ell)$. Serre showed that 
\begin{equation} \label{serre}
[\GL(2, \Z_\ell) : G_\infty] < \infty, 
\end{equation}
with the index being $1$ for all but finitely many $\ell$. The proof of Serre's theorem relies on the relative paucity of subgroups of $\GL(2, \Z_\ell)$. In our dynamical setting, on the other hand, one finds that $\Aut(T_\infty)$ has a discouraging abundance of subgroups; for instance when $d = 2$, every countably based pro-2 group is a subgroup of $\Aut(T_\infty)$, and matters are at least as bad for larger $d$. Nonetheless, some techniques are available for showing that $G_\infty(\phi)$ must be a large subgroup of $\Aut(T_\infty)$ in certain cases, and we survey them and the results they provide in Section \ref{generic}. In addition, we provide some evidence supporting the idea that Question \ref{imagequestion}(a) has an affirmative answer in general.

Question \ref{imagequestion}(a) does not have a positive answer for all $\phi$, just as Serre's theorem does not hold when $E$ has complex multiplication, but in attempting to make a precise conjecture one encounters serious obstacles in locating the cases that must be excluded.  In Section \ref{exceptional}, we discuss in some detail four types of these exceptional maps, including those that are post-critically finite (see p. \pageref{pcfref} for a definition) and those that commute with a non-trivial M\"obius transformation. In the case where $\phi$ is quadratic, enough results and examples have now been accumulated that we conjecture these four types constitute the only exceptions (see Conjecture \ref{quadratic}). 

To prove zero-density theorems for primes dividing a given orbit of $\phi$, one does not need information as strong as $[\Aut(T_\infty) : G_\infty(\phi)] < \infty$. This raises our second primary question:
\begin{question} \label{densityquestion} 
Let $K$ be a global field. For which maps $\phi \in K(x)$ can we deduce enough about $G_\infty(\phi)$ to ensure that the limiting proportion of fixed points given in \eqref{thekey0} is zero, and hence
all orbits of $\phi$ have density zero prime divisors?
\end{question}
In Section \ref{density}, we survey results showing that in some cases minimal information about $G_\infty(\phi)$ suffices. These results proceed via a possibly unexpected use of the theory of stochastic processes, and they lead to a variety of zero-density theorems (see Theorem \ref{sumup} for an example). 
Many of the results in Sections \ref{generic} and \ref{density} rely on being able to establish that the numerators of $\phi^n$ are irreducible for all $n \geq 1$. Results in this direction, which are of interest in their own right, are surveyed in Section \ref{stability}. 




\section{The image of $\rho$: generic case} \label{generic}

\subsection{A tour of known results} \label{tour}
Let $\rho, G_\infty(\phi),$ and $G_n(\phi)$ be defined as on p. \pageref{rhodef}. Any discussion of the generic situation must begin with work of Odoni, who in \cite{odonigalit} studied the case where $K$ is a field of characteristic zero, $t_0, \ldots, t_{d-1}$ are algebraically independent over $K$, and 
\begin{equation} \label{genmon}
F(x) = x^d + t_{d-1}x^{d-1} + \cdots + t_1x + t_0.
\end{equation}
Let $T_\infty$ be defined as in \eqref{tdef} with $\phi = F$; now it resides in the algebraic closure of $K(t_0, \ldots, t_{d-1})$. Odoni shows \cite[Theorem I]{odonigalit}: 
\begin{theorem}[\cite{odonigalit}] \label{odonigen}
With notation as above, $G_\infty(F) = \Aut(T_\infty)$.
\end{theorem}Á
In the case where $K$ is a number field, one may then fix $n$ and apply Hilbert's irreducibility theorem to deduce that $G_n(f) = \Aut(T_n)$ for all but a ``thin set" $E_n$ of degree-$d$ polynomials $f$ defined over $K$. Unfortunately, $E_n$ is not effectively computible, and moreover one cannot rule out that the union of the $E_n$ includes all degree-$d$ polynomials defined over $K$. Indeed, Odoni makes the following tentative conjecture, which is a special case of \cite[Conjecture 7.5]{odonigalit}:
\begin{conjecture}[Odoni]
For each $d \geq 2$, there exists a monic polynomial $f(x) \in \Z[x]$ of degree $d$ with $G_\infty(f) = \Aut(T_\infty)$. 
\end{conjecture}
This conjecture remains open for all $d \geq 3$. While Theorem \ref{odonigen} does not answer Question \ref{imagequestion} (a) or (b) for any single polynomial, it does offer evidence that in the absence of some sort of arithmetic coincidence one expects to find $G_\infty(f) = \Aut(T_\infty)$. 

Quadratic polynomials, as they do often in questions related to dynamics, furnished the first realm where it became possible to answer Question \ref{imagequestion} in certain cases, though much remains unknown. As noted in the introduction, Odoni showed in \cite{odoniwn} that $G_\infty(f) = \Aut(T_\infty)$ in the case $K = \Q$, $f(x) = x^2 - x + 1$.  Moreover, in \cite[Section 4]{odoni}, he responded to a question of J. McKay by giving a powerful algorithm for deciding whether $G_n(f) = \Aut(T_n)$ for $f(x) = x^2 + 1$. J. Cremona \cite{cremona} used this algorithm\footnote{According to Cremona, his ability to push the calculation so far relied in part on a computer bug. While running his program on a powerful computer cluster at the University of Bath, a Friday night glitch effectively killed all the processes but his, allowing his program to hog the machine all weekend.} to verify the assertion for $n$ up to $5 \cdot 10^7$. Note that for $n = 5 \cdot 10^7$,  
$$\log_2 |\Aut(T_n)| = 32^{10000000} - 1,$$ showing that Odoni's method goes far beyond what brute force computation could allow. M. Stoll \cite{stoll} then furnished a clever trick to show that Odoni's algorithm works for all $n$, and generalized the result to many other cases:
\begin{theorem}[\cite{stoll}] \label{stollthm} Let $K = \Q$ and $f(x) = x^2 + k \in \Z[x]$, where $-k$ is not a square, and one of the following holds:
\begin{itemize}
\item $k > 0, k \equiv 1 \bmod{4}$
\item $k > 0, k \equiv 2 \bmod{4}$
\item $k < 0, k \equiv 0 \bmod{4}$
\end{itemize}
Then $G_\infty(f) = \Aut(T_\infty)$.
\end{theorem}
Interestingly, there is no way to extend Stoll's method to all other cases where $-k$ is not a square. For instance, when $k = 3$, one finds that $[\Aut(T_3) : G_3(f)] = 2$, even though the third iterate of $f(x) = x^2 + 3$ is irreducible. As we will see shortly, this arises from the curious fact that both $f^2(0)$ and $f^3(0)$ have large square factors. W. Hindes \cite{hindes} has recently shown that $k=3$ is the only integer to exhibit this particular degeneracy, thereby answering a question of the author. In \cite{hindes2}, Hindes conjectures that $[\Aut(T_{\infty}): G_\infty(f)]=2$, using Ê
an updated form of Hall's conjecture (involving the size of the integral points on the Mordell curves $y^2=x^3+A$). However, at present it is not known whether $[\Aut(T_{\infty}): G_\infty(f)]$ is even finite. \label{3case}

In addition to these polynomial cases, there is at present just one more rational function $\phi \in \Q(x)$ for which it is known that $G_\infty(f) = \Aut(T_\infty)$ for $K = \Q$, namely 
\begin{equation} \label{special}
\phi(x) = \frac{1+3x^2}{1-4x-x^2}.
\end{equation}
See \cite[Theorem 1.2]{galrat}. This particular function has a critical point at $x = 1$ that lies in a two-cycle, similar to a polynomial's fixed critical point at infinity. We shall have more to say at the end of Section \ref{methods} about the additional fortuitous properties of $\phi$ that allow for this result. In \cite[Theorem 3.2]{thesis}, it is shown using a minor variation of Stoll's technique that we have $G_\infty(f) = \Aut(T_\infty)$ for $f(x) = x^2 + t$, provided that $K$ has characteristic $p \equiv 3 \bmod{4}$. It would be interesting to know if the same results holds when $K$ has arbitrary odd characteristic.


In a handful of additional cases it is known that $[\Aut(T_\infty) : G_\infty(f)] < \infty$. The next result follows from work in \cite{quaddiv}; see the remark following the proof of Theorem 1.1 of \cite{quaddiv}. We recall some definitions:
\begin{itemize}
\item A rational map is \textit{post-critically finite} \label{pcfref} if the forward orbit of each of its critical points is finite (see Section \ref{pcf} for more about such maps).
\item A point $\alpha$ is \textit{periodic} under a rational map $\phi$ if $\phi^n(\alpha) = \alpha$ for some $n \geq 1$.
\item A point $\alpha$ is \textit{pre-periodic} under $\phi$ if $\phi^n(\alpha) = \phi^m(\alpha)$ for some $n > m \geq 0$, where we set $\phi^0(\alpha) = \alpha$. 
\item A point $\alpha$ is \textit{strictly pre-periodic} under $\phi$ if it is pre-periodic but not periodic.
\end{itemize}

\begin{theorem}[\cite{quaddiv}] \label{quadfin}
Let $K = \Q$, and $f \in \Z[x]$ be monic and quadratic.  Suppose $f$ is not post-critically finite, and $0$ is strictly pre-periodic under $f$.  Assume further that all iterates of $f$ are irreducible over $\Q$.  Then 
$G_\infty$ has finite index in $\Aut(T_\infty)$.  
\end{theorem}  

The irreducibility hypothesis in Theorem \ref{quadfin} is essential, as will be shown in Section \ref{methods}. It is tempting to replace it with the condition that  the number of irreducible factors of $f^n(x)$ be bounded independently of $n$, but at present no proof is known with this weaker condition. As an illustration, suppose that $f(x)$ splits into two linear factors $g_1(x)g_2(x)$, but $g_1(f^{n}(x))$ and $g_2(f^{n}(x))$ are irreducible for all $n \geq 1$ (such statements are often provable; see for example the discussion of eventual stability in Section \ref{stability} and \cite[Proposition 4.5]{quaddiv}). Even if one computes the Galois groups of $g_1(f^{{n-1}}(x))$ and $g_2(f^{n-1}(x))$, which is also often possible, one must then address the possibility that these groups do not operate independently on the roots of $f^n(x)$. In other words, the splitting fields of $g_1(f^{n-1}(x))$ and $g_2(f^{n-1}(x))$ may have non-trivial intersection. Getting a handle on this intersection appears to be a difficult problem. 

In \cite{quaddiv}, it is shown that Theorem \ref{quadfin} applies to these families of maps:   
\begin{enumerate}
\item $f(x) = x^2 - kx + k$ for all $k \in \Z \setminus \{-2, 0, 2, 4\}$.  
\item   $f(x) = x^2 + kx - 1$ for all $k \in \Z \setminus \{-1, 0, 2\}$. 
\end{enumerate}
In family (1), the exceptions $k = -2, 0, 2$ give polynomials that are post-critically finite, while for $k = 4$ we obtain the reducible polynomial $g(x) = x^2 - 4x + 4$. In \cite[Proposition 4.6]{quaddiv} it is shown that $g^n(x)$ is the square of an irreducible polynomial for all $n \geq 1$, and this is enough to allow for a density zero result for orbits of this map (see Section \ref{density}). However, at present no proof that $[\Aut(T_\infty) : G_\infty(g)] < \infty$ is known. In family (2), $k = 0, 2$ give post-critically finite polynomials, while for $k = -1$ the polynomial $h(x) = x^2 - x - 1$ 
\label{cooled} has the curious property that $h^2(x)$ is irreducible, but $h^3(x)$ factors as the product of two irreducible quartics. This furnishes the same obstacles to showing $[\Aut(T_\infty) : G_\infty(h)] < \infty$ as in the $k = 4$ case for family (1). Interestingly, $h(x)$ is the minimal polynomial of the golden mean. No one knows whether special properties of the golden mean are related to the highly unusual factorization of $h^3(x)$. 

Recently, C. Gratton, K. Nguyen, and T. Tucker \cite{abc} proved another important result in this area, giving evidence that one should expect $[\Aut(T_\infty) : G_\infty(f)] < \infty$ when $f$ is a quadratic polynomial. 
\begin{theorem}[\cite{abc}] \label{abc}
Let $K = \Q$, and let $f(x) \in \Z[x]$ be monic, quadratic, and not post-critically finite. Assume that all iterates of $f$ are irreducible. Then the ABC conjecture implies $[\Aut(T_\infty) : G_\infty(f)] < \infty$.
\end{theorem}
Theorem \ref{abc} is a slightly generalized form of \cite[Proposition 6.1]{abc}; we explain below how it follows from the main results of \cite{abc}. With minimal difficulty, one can generalize Theorem \ref{abc} so that the field of definition of $f$ is a number field. However, as in Theorem \ref{quadfin}, the irreducibility hypothesis on the iterates of $f$ is essential. Thus we are left with the surprising state of affairs that establishing the irreduciblility of iterates of $f$ is the key step; once that is known, it is an easier path to prove the Galois groups of such iterates are large. We will see this theme again when examining rational functions with non-trivial automorphisms in Section \ref{autom}. See Section \ref{stability} for more on the question of irreducibility of iterates. 

\subsection{A sketch of the method} \label{methods}
Let us briefly sketch the method underlying the results of Section \ref{tour}, restricting ourselves to the situation where $K = \Q$ and $f \in \Z[x]$ is a monic, quadratic polynomial. This case is of sufficient simplicity to highlight the essential elements of the method, but of sufficient depth to require much of their full strength. We denote by $c$ the critical point of $f$, and assume that $c$ does not lie in $f^{-n}(0)$ for any $n \geq 0$, thereby ensuring that $f^n(x)$ has $2^n$ distinct roots. Let $K_n$ denote the field $\Q(f^{-n}(0))$, so that $G_n = \Gal(K_n/\Q)$, and denote by $H_n$ the Galois group of the relative extension $K_{n}/K_{n-1}$ \label{hndef}. One may roughly summarize the method as follows: $H_n$ is as large as possible provided that a prime ramifies in the extension $K_n/K$ that did not already ramify in $K_{n-1}/K$. Candidates for this newly ramified prime are found only among primes dividing $f^n(c)$ that do not divide $f^i(c)$ for $i < n$, and thus we must study the arithmetic of the orbit of $c$ under $f$. Sufficient knowledge of this arithmetic is available only in the cases covered by the results in Section \ref{tour}. 

Note that $K_{n}$ is obtained from $K_{n-1}$ by adjoining the roots of $f(x) - \beta_i$, where $\beta_1, \ldots \beta_{2^{n-1}}$ are the roots of $f^{n-1}(x)$. This is the same as adjoining the $2^{n-1}$ square roots
$\sqrt{\delta_i}$, where
$$\delta_i := \Disc(f(x) - \beta_i),$$
and thus $K_{n}$ is a $2$-Kummer extension of $K_{n-1}$, and we have an injection $H_n \hookrightarrow (\Z/2\Z)^{2^{n-1}}$. This injection is also apparent from our identification of $G_n$ with a subgroup of $\Aut(T_n)$, since $H_n$ must lie in the kernel of the restriction mapping $\Aut(T_n) \to \Aut(T_{n-1})$, which is generated by the transpositions swapping a pair of vertices at level $n$ that are both connected to the same vertex at level $n-1$. We thus refer to $H_n$ as \textit{maximal} when 
$$
\text{$H_n = \ker(\Aut(T_n) \to \Aut(T_{n-1}))$, \hspace{0.1 in} or equivalently $H_n \cong (\Z/2\Z)^{2^{n-1}}$.}
$$
Clearly we have $G_n(f) = \Aut(T_n)$ if and only if $H_i$ is maximal for $i = 1, 2, \ldots, n$. 

Using Kummer theory (e.g. \cite[Section VI.8]{langalg}), $[K_{n} : K_{n-1}]$ is the order of the group $D$ generated by the classes of the $\delta_i$ in 
$K_{{n-1}}^*/K_{n-1}^{*2}$, where $K_{n-1}^{*2}$ denotes the non-zero squares in $K_{n-1}$.  We have
\[
\#D = \frac{2^{2^{n-1}} }{ \#V}, 
\text{ where }
V = \{(e_1, \ldots, e_{2^{n-1}}) \in 
\mathbb{F}_2^{2^{n-1}} : \prod_j \delta_j^{e_j} \in K_{n-1}^{*2} \}.
\]
Thus $V$ is the group of multiplicative relations among the $\delta_i$, up to squares.  One sees easily that $V$ is an $\mathbb{F}_2$-vector space, and that the action of $G_{n-1}$ on the $\delta_i$ gives an action of $G_{n-1}$ on $V$ as linear transformations.  It follows that $V$ is an $\mathbb{F}_2[G_{n-1}]$-module. Perhaps surprisingly, one can show that if $V$ is non-trivial, then it must contain the element $(1, \ldots, 1)$ \textit{provided that the action of $G_{n-1}$ on the $\delta_i$ is transitive}, or equivalently that $f^{n-1}(x)$ is irreducible. One begins by showing that if $V \neq 0$, then the submodule $V^{G_{n-1}}$ of $G_{n-1}$-invariant elements is non-trivial, a result that relies on $G_{n-1}$ being a $2$-group (see \cite[Lemma 1.6]{stoll})). The transitivity of the action of $G_{n-1}$ on the $\delta_i$ then assures that if $V^{G_{n-1}}$ is non-empty, then it must contain $(1, \ldots, 1)$. 

Now $(1, \ldots, 1) \in V$ if and only if 
\begin{equation} \label{discprod}
\prod_{i = 1}^{2^{n-1}} \Disc(f(x) - \beta_i)
\end{equation}
is a square in $K_{n-1}$. But $\Disc(f(x) - \beta_i) = -4(b - \beta_i)$, where we write $f(x) = (x- c)^2 + b$. As the $\beta_i$ vary over all roots of $f^{n-1}(x)$, the product in \eqref{discprod} is $(-4)^{2^{n-1}}f^{n-1}(b)$. But $f^{n-1}(b) = f^{n-1}(f(c)) = f^n(c)$, and hence \eqref{discprod} is a square in $K_{n-1}$ if and only if $f^n(c)$ is a square in $K_{n-1}$. 

To sum up, assuming that $f^{n-1}(x)$ is irreducible, we've shown 
\begin{equation} \label{onecrit}
\text{$[K_n : K_{n-1}] = 2^{2^{n-1}}$ if and only if $f^n(c)$ is not a square in $K_{n-1}$}.
\end{equation} 
This key result has generalizations in a variety of directions. An easy and direct generalization is to replace the ground field $\Q$ with any number field $K$ (and allow $K$ to be the field of definition for $f$). For a similarly small price, one can let $G_n$ be the Galois group over $K$ of polynomials of the form $g(f^n(x))$, where $f$ is still quadratic and $g$ is arbitrary \cite[Lemma 3.2]{quaddiv}. When $f$ is allowed to be a quadratic rational function, the only known result becomes significantly more complicated: the condition is essentially that the numerator of $f^n(c_1)f^n(c_2)$ not be a square in $K_{n-1}$, where $c_1$ and $c_2$ are the two critical points of $f$ \cite[Theorem 3.7]{galrat}. 

To apply these results in any of the above settings requires showing that a given element of $K_{n-1}$ is not a square in $K_{n-1}$, a problem which seems difficult at first blush since $K_{n-1}$ is generally a huge-degree extension. However, the element in question (e.~g.~$f^n(c)$) is in fact an element of the ground field, which makes things considerably easier. Let us return to the setting where $f(x)$ is a monic, quadratic polynomial defined over $\Z$, and our ground field is $\Q$. If $f^n(c)$ is divisible to an odd power by a prime $p$, then $f^n(c)$ can only become a square in $K_{n-1}$ if $p$ ramifies in $K_{n-1}$. The iterative nature of the extensions $K_{n-1}$ allows us to explicitly describe a set of primes that must include all those that ramify in $K_{n-1}$. More specifically, a calculation with resultants gives
\begin{equation} \label{disc}
\Disc(f^k) = \pm 2^{2^k} (\Disc(f^{k-1}))^2  f^k(c)  
\end{equation}
for all $k \geq 1$ \cite[Lemma 2.6 and discussion following]{quaddiv}. The appearance of $f^k(c)$ in \eqref{disc} is actually rather intuitive: $f^k$ has a multiple root modulo an odd prime $p$ only when $f^{k-1}$ already had such a root, or a critical point appears in $f^{-k}(0)$ modulo $p$. The latter condition is equivalent to $f^k(c) \equiv 0 \bmod{p}$, or $p \mid f^k(c)$. Because $K_{n-1}$ is the splitting field of $f^{n-1}(x)$ over $\Q$, the true discriminant of the extension $K_{n-1}/\Q$ divides $\Disc(f^{n-1})$. A simple induction using \eqref{disc} gives that the only primes dividing $\Disc(f^{n-1})$ are those dividing one of $2, f(c), f^2(c), \ldots, f^{n-1}(c)$. We at last obtain the criterion that gives rise to nearly all of the results of Section \ref{tour}:
\begin{theorem}[\cite{quaddiv}] \label{criterion}
Let $f \in \Z[x]$ be monic and quadratic with critical point $c$, and let $K_n$ and $H_n$ be defined as on p. \pageref{hndef}. Assume that $f^{n-1}(x)$ is irreducible and there exists an odd prime $p \in \Z$ whose $p$-adic valuation $v_p$ satisfies $v_p(f^n(c))$ odd and $v_p(f^i(c)) = 0$ for $i = 1, 2, \ldots, n-1$. Then $H_n$ is maximal. 
\end{theorem}
In other words, assuming that $f^{n-1}(x)$ is irreducible, the element $f^n(c)$ of the sequence $(f^i(c))_{i \geq 1}$ must have a primitive prime divisor appearing to odd multiplicity. In terms of ramification, Theorem \ref{criterion} requires that a ``new" prime $p$ ramify in $K_n$ (that is, one that has not already ramified in $K_{i}$ for $i < n$). This result has generalizations in the same directions as those of \eqref{onecrit}; see \cite[Theorem 3.3]{quaddiv} and \cite[Corollary 3.8]{galrat}. 

The hypothesis that $f^{n-1}(x)$ be irreducible is essential in Theorem \ref{criterion}. Fortunately, in many cases one finds that \textit{all} iterates of $f(x)$ are irreducible, a fact we discuss further in Section \ref{stability}. Indeed, to show this it is enough to prove that $f$ is irreducible and the orbit of $c$ under $f$ (called the \textit{critical orbit} of $f$) contains no squares (see Theorem \ref{stab}). This fact, together with Theorem \ref{criterion}, shows that the nature of $G_\infty(f)$ depends crucially (critically, even) on arithmetic properties of the critical orbit of $f$.
This makes for a striking analogy with complex and real dynamics, where analytic properties of the critical orbit of a quadratic polynomial have been shown to determine fundamental dynamical behavior of the polynomial.  For instance, if $f \in \mathbb{C}[z]$ is quadratic, then membership in the Mandelbrot set -- and equivalently the connectedness of the filled Julia set of $f$ -- is determined by whether the critical orbit remains bounded \cite[Section 3.8]{Devaney}.  

To apply Theorem \ref{criterion} requires getting a handle on the primes dividing elements in the critical orbit of $f$, which is generally very difficult. One may obtain some tantalizing results, however. In Section \ref{density} we will see that it is vital to be able to show that $H_n$ is maximal for infinitely many $n$. We invite the reader to show that in the setting of Theorem \ref{criterion} there are infinitely many $n$ such that $f^n(c)$ has a primitive prime divisor\footnote{or you can take the easy way out and look at \cite[Theorem 6.1]{quaddiv}.}; unfortunately one cannot guarantee that the first appearance of such a prime in the sequence $(f^i(c))_{i \geq 1}$ is to odd multiplicity. Similarly, one can show that there must be infinitely many primes $p$ dividing at least one term $f^n(c)$ to odd multiplicity; unfortunately, one cannot guarantee that when they do so their appearance is primitive. 

The ABC conjecture rescues us from this predicament, as shown in \cite[Theorem 1.2]{abc}: it implies that the for all but finitely many $n$, there is a primitive prime divisor of $f^n(c)$ appearing to multiplicity 1 (note that in our situation, $f(x)$ is dynamically ramified in the terminology of \cite{abc} if and only if $f(x) = x^2$). Thus Theorem \ref{abc} is an immediate corollary of \cite[Theorem 1.2]{abc}. Interestingly, J. Silverman \cite{jhszig} has shown that in higher dimensions, Vojta's conjecture implies a result on primitive prime divisors similar to \cite[Theorem 1.2]{abc}. However, as the Galois theory of preimages in the higher-dimensional setting is all but nonexistent at present, the Galois-theoretic implications of Silverman's result remain unclear. 

In special circumstances, we may even obtain unconditional results. 
When $f(x) = x^2 + k$, an easy application of Theorem \ref{stab} shows that all iterates of $f$ are irreducible provided that $-k$ is not a square. The lack of linear term in $f$ ensures that the resulting critical orbit $(f^i(0))_{i \geq 1}$ satisfies a powerful property known as \textit{rigid divisibility} \cite[p. 524]{quaddiv}. Namely, setting $a_n = f^n(c)$ we have:
\begin{itemize}
\item $v_p(a_n) > 0$ implies $v_p(a_{mn}) = v_p(a_n)$ for all $m \geq 1$, and 
\item $p^e \mid a_n$ and $p^e \mid a_m$ implies $p^e \mid a_{\gcd(m,n)}$. 
\end{itemize}
One then defines a ``primitive part" $b_n$ of each $a_n$ by setting $b_n = \prod_{d \mid n} a_n^{\mu(n/d)}$, where $\mu$ denotes the M\"obius function, and shows that the $b_n$ are pairwise relatively prime. By Theorem \ref{criterion}, to prove that $H_n$ is maximal then only requires showing that $b_n$ is not a unit times a square. This is Odoni's criterion, used by Cremona in \cite{cremona} and by Stoll to prove Theorem \ref{stollthm}. We now may shed some light on the case of $f(x) = x^2 + 3$: we have
$f^2(0) = 2^2 \cdot 3$, and $f^3(0) = 7^2 \cdot 3$, whence $b_2 = 2^2$ and $b_3 = 7^2$. Note that $f^2(0)$ is not a square in $K_1 = \Q(\sqrt{-3})$, and hence $G_2(f) = \Aut(T_2)$ by \eqref{onecrit}, but $f^3(0)$ is a square in $K_2 = \Q(f^{-2}(0))$, and hence $G_3(f) \neq \Aut(T_3)$. 

Another favorable case occurs when $f$ maps $0$ into a cycle not containing $0$ (or in other words, $0$ is strictly pre-periodic under $f$). Then we have a finite set $R$ consisting of all primes dividing at least one of the elements in the orbit of $0$. If $p$ is not such a prime, then $p \mid f^n(c)$ implies $f^{m+n}(c) \equiv f^m(f^n(c)) \equiv f^m(0) \not\equiv 0 \bmod{p}$ for any $m \geq 1$, and hence $p$ divides at most one element of the critical orbit of $f$. An extreme example of this situation comes from the polynomial $f(x) = x^2 - x + 1$ considered by Odoni in \cite{odoniwn}. Here $f$ sends $0$ to the fixed point $1$, and thus $R$ is empty and the elements of the critical orbit (considered as belonging to $\Z[\frac{1}{2}])$ are pairwise relatively prime. To show $G_\infty(f) = \Aut(T_\infty)$ via Theorem \ref{criterion}, the challenge is then to show that $f^n(c)$ is not a square for any $n \geq 1$, which proves surprisingly difficult \cite[p. 3]{odoniwn}.

In the more general situation when $0$ is pre-periodic under $f$, to apply Theorem \ref{criterion} it suffices to show that $f^n(c)$ is not a square times a (possibly empty) product of primes in $R$. One may appeal to Siegel's theorem to achieve this, for the price of excluding a finite set of $n$. Indeed, if $f^n(c)$ is $r$ times a square, for some product $r$ of primes in $R$, then the curve $ry^2 = f(f(x))$ has an $S$-integral point with $x = f^{n-2}(c)$ (here $S$ is empty if $c \in \Z$ and $S = \{2\}$ if $c \not\in \Z$). But there can be only finitely many such points, for each of the finitely many choices of $r$. This establishes Theorem \ref{quadfin}. 

Finally, let us return to the rational function $\phi$ given in \eqref{special}. As noted on p. \pageref{special}, $\phi$ has a critical point at $x = 1$ that lies in a two-cycle, making $\phi$ similar to a polynomial. Moreover, $\phi$ sends $0$ into the two-cycle $1 \to -1 \to 1$, ensuring that we are in the situation of the previous two paragraphs, and even better with $R = \emptyset$. It follows that any odd prime divides the numerator of at most one term of the wandering critical orbit $\{\phi^n(-1/3) : n \geq 1\}$, and the same is true of the sequence $(a_n) := (p_n(-1/3)p_n(1) : n \geq 1)$, where $p_n(x)$ is essentially the numerator of $\phi^n(x)$ (see \cite[Section 2]{galrat} for a precise definition). The numbers $p_n(-1/3)p_n(1)$ play the role of $f^n(c)$ in Theorem \ref{criterion} (see \cite[Corollary 3.8]{galrat}), though in Theorem \ref{criterion} we required that the desired prime $p$ be odd, and here we require it to be odd and not equal to 3.
By reducing modulo 5, one shows that no element of $(a_n)$ is plus or minus a square. Hence each element of $(a_n)$ is divisible by some prime to odd multiplicity, and if this prime is not two or three then
its appearance is primitive and we may apply the equivalent of Theorem \ref{criterion}. The proof that $[\Aut(T_\infty) : G_\infty(\phi)] < \infty$ thus finishes with a calculation showing the evenness of the $2$-adic and $3$-adic valuation of all terms of $(a_n)$.  See the end of Section 3 of \cite{galrat} for the full details. 

A major obstacle to extending these methods to higher-degree polynomials is that in \eqref{disc} and Theorem \ref{criterion}, the appearance of $f^n(c)$ is replaced by $\prod f^n(c)$, where the product is over all critical points $c$ of $\phi$. There appears to be no easy way to rule out arithmetic interactions among the elements of several critical orbits.

\section{The image of $\rho$: exceptional cases} \label{exceptional}

In light of the results of Section \ref{generic}, and especially Theorems \ref{odonigen} and \ref{abc}, it is tempting to conjecture that $[\Aut(T_\infty) : G_{\infty}(\phi)] < \infty$ unless there is a structural reason this cannot occur.  In the setting of Galois representations attached to elliptic curves, the structural reason is the curve having complex multiplication, and Serre's theorem (see \eqref{serre}) shows that this is the only exception. In our case, one encounters a profusion of structural reasons, four of which we discuss in this section. Unfortunately, there does not seem to be a general principle to suggest that these four exhaust all possibilities, and correspondingly it seems impossible at present to make a convincing finite-index conjecture. However, enough results and examples have now been accumulated for quadratic $\phi$ that we pose a conjecture in this case: see Conjecture \ref{quadratic}.

We begin with examples of four rational functions for which $[\Aut(T_\infty) : G_\infty(\phi)]$ is infinite, each illustrating a class of exceptions to any finite index conjecture:
\begin{enumerate}
\item[(a)] $\phi(x) = x^2 - 2$
\item[(b)] $\phi(x) = x^3 + 2$
\item[(c)] $\phi(x) = x^2 + x$
\item[(d)] $\phi(x) = (x^2 + 1)/x$
\end{enumerate}
In (a), $\phi$ is post-critically finite, which we recall means that the forward orbit of each critical point of $\phi$ is finite.  In (b), $\phi$ is not post-critically finite, but has overlapping critical orbits: $0$ is a point of multiplicity 3, which may be thought of as two co-incident critical orbits. In (c), the root $0$ of $T_\infty$ is periodic under $\phi$. In (d), $\phi$ commutes with a non-trivial M\"obius transformation. 

In cases (a) and (b), the impediments to $[\Aut(T_\infty) : G_\infty(\phi)]$ being finite are geometric, in that they are invariant under changing the root of the tree $T_\infty$. Cases (c) and (d), on the other hand, are arithmetic in that they depend on the root of $T_\infty$ having special algebraic properties. In Section \ref{pcf} we discuss the case where $\phi$ is post-critically finite, and we show $G_\infty(\phi)$ has infinite index in $\Aut(T_\infty)$ for such maps. We also discuss what is known about the group $G^{\text{arith}}$ \label{aimg} obtained by replacing the root $0$ of $T_\infty$ by an element $t$ that is transcendental over $K$, and working over the ground field $K(t)$. This group gives an over-group for $G_\infty(\phi)$, and the latter is obtained by the specialization $t = 0$. 
 Cases (b), (c), and (d) are discussed in Sections \ref{overlap}, \ref{pdc}, and \ref{autom}, respectively.

\subsection{Post-critically finite rational functions} \label{pcf}

The discussion just before and after Theorem \ref{criterion} shows how the arithmetic of the forward orbit of the critical point of a quadratic polynomial $\phi$ plays a key role in the study of $G_\infty(\phi)$. A similar relationship holds for more general maps, as we now explain. 
Let $K_n = K(\phi^{-n}(0))$, and consider the question of which primes of $K$ ramify in $K_n$, and in particular which ramify in $K_n/K$ but not in $K_{n-1}/K$. Thanks to several generalizations of the discriminant formula \eqref{disc}, it is known that these primes must belong to a very restricted set. First W. Aitken, F. Hajir, and C. Maire gave a generalization to polynomials of arbitrary degree \cite{hajir}, and recently J. Cullinan and Hajir \cite{HC} as well as the author and M. Manes  \cite[Theorem 3.2]{galrat} have produced further generalizations to rational functions. In each case, the formulae show that the only primes of $K$ that can ramify in the extensions $K_n/K$ are those dividing $\phi^i(c)$ for some critical point $c$ of $\phi$ (aside from a finite set of primes that does not grow with $n$, such as the primes dividing the resultant of $\phi$). 

Now a generic rational function $\phi \in K(x)$ of degree $d$ has $2d - 2$ distinct critical points, all with infinite and non-overlapping orbits. This allows for the collection of primes dividing at least one element of the form $\phi^i(c)$ to be large, and thus there are many possibilities for primes ramifying in $K_n$. At the extreme of non-generic behavior are the post-critically finite rational functions, which have only a \textit{finite set} of primes, independent of $n$, that can ramify in any $K_n$ (this is among the main results in \cite{hajir} and \cite{HC}). In other words, the extension $K_\infty := \bigcup_{n = 1}^\infty K_n$ is a finitely ramified extension of $K$.

Because the inertia subgroups at the ramified primes generate $\Gal(K_\infty/U_\infty)$, where $U_\infty$ is the (presumably small) maximal unramified sub-extension of $K_\infty$, we should generally expect $G_\infty(\phi)$ to be a small subgroup of $\Aut(T_\infty)$ when $\phi$ is post-critically finite. We now give a result in this direction, whose proof evolved through discussions between the author and R. Pink.
\begin{theorem} \label{infindex}
Suppose that $K$ is a global field of characteristic $0$ or $> d$, and let $\phi \in K(x)$ be a post-critically finite map of degree $d$. Then $[\Aut(T_\infty) : G_\infty(\phi)]$ is infinite. 
\end{theorem}

\begin{proof}
We first argue that if $H \leq \Aut(T_\infty)$ is (topologically) generated by the conjugacy classes of finitely many elements, then $[\Aut(T_\infty) : H]$ is infinite. A standard result in group theory is that $\Aut(T_n) \cong S_d^{(n)}$, where the latter group is the $n$-fold iterated wreath product of the symmetric group $S_d$ on $d$ letters. Moreover, the abelianization of $\Aut(T_n)$ is given by
\begin{equation} \label{abseq}
\Aut(T_n)^{\text{ab}} \cong ((S_d)^{\text{ab}})^n \cong (\Z/2\Z)^n,
\end{equation}
where the first isomorphism follows from the fact that the abelianization of the the wreath product of groups $G_1$ and $G_2$ is $G_1^{\text{ab}} \times G_2^{\text{ab}}$ \cite[p. 215]{tggt}. Denote by 
$$\tau : \Aut(T_\infty) \twoheadrightarrow (\Z/2\Z)^{\mathbb{N}}$$ 
the homomorphism obtained from \eqref{abseq}.  By our assumption about $H$, the group $\tau(H)$ is finitely generated, and hence finite. Therefore $[\tau(\Aut(T_\infty)) : \tau(H)]$ is infinite, whence $[\Aut(T_\infty) : H]$ is infinite as well. 



By the main result of \cite{HC}, the extension $K_\infty$ of $K$ is unramified outside a finite set $S$ of places of $K$. Assume first that $K$ is a number field.  It follows from a result of Ihara (see \cite[Theorem 10.2.5]{nsw}) that the Galois group $G_{K,S}$ of the \textit{maximal} extension of $K$ unramified outside $S$ is (topologically) generated by the conjugacy classes of finitely many elements. As $G_\infty(\phi)$ is a quotient of this group, it shares the same property. 
When $K$ is a global function field, the group $G_{K,S}$ may be quite complicated in general. However, our assumption that $K$ has characteristic $> d$ implies that the ramification in $K_\infty$ is tame, and hence $G_\infty(\phi)$ is a quotient of the maximal tamely ramified extension of $K$ that is unramified outside $S$. This latter group is (topologically) finitely generated \cite[Corollary 10.1.6]{nsw}, and hence so is $G_\infty(\phi)$. In particular, $G_\infty(\phi)$ is (topologically) generated by the conjugacy classes of finitely many elements. 
\end{proof}

Let us now discuss the group $G^{\text{arith}}$, first alluded to on p. \pageref{aimg}. Let $t$ be transcendental over $K$, and consider the action of the absolute Galois group of $K(t)$ on the tree $T_\infty(t) \subset \overline{K(t)}$ of iterated pre-images of $t$ under $\phi$. The image of this action is $G^{\text{arith}}$, also known as the \textit{(profinite) arithmetic iterated monodromy group} of $\phi$. Note that $\Aut(T_\infty(t))$ and $\Aut(T_\infty)$ are naturally isomorphic, so we may think of $G_\infty(\phi)$ as the subgroup of $G^{\text{arith}}$ obtained via the specialization $t = 0$. Thus in a loose sense $G^{\text{arith}}$ gives the group one expects $G_\infty(\phi)$ to be under the choice of a generic root of $T_\infty$.



As $G^{\text{arith}}$ is the Galois group of the extension $K_{\infty ,t} := \bigcup_{n \geq 1} K(\phi^{-n}(t))$ over $K(t)$, it has a normal subgroup $G^{\text{geom}}$ corresponding to the subfield
$L(t)$, where $L := \overline{K} \cap K_{\infty, t}$ is the maximal constant field extension contained in $K_{\infty ,t}$. This gives an exact sequence
\begin{equation} \label{sequence}
1 \to G^{\text{geom}} \to G^{\text{arith}} \to \Gal(L/K) \to 1,
\end{equation}
The primes of $L(t)$ over which $K_{\infty, t}$ is ramified correspond to the ramification points of the covers $\phi^n : \mathbb{P}^1 \to \mathbb{P}^1$, for $n = 1, 2, \ldots$. One easily sees that this is the same as the post-critical set $C$ of $\phi$, namely the set $\{\phi^n(c) : \text{$n \geq 1$ and $c$ is a critical point of $\phi$}\}$.  In the case where the characteristic of $K$ is either 0 or greater than the degree $d$ of $\phi$, the extension $K_{\infty, t}$ has only tame ramification over $L(t)$, and hence $G^{\text{geom}}$ is a quotient of the tame fundamental group of $\mathbb{P}^1_{\overline{K}} \setminus C$. When $\phi$ is post-critically finite, the resulting finiteness of $C$ implies that this tame fundamental group is (topologically) finitely generated, and hence so is $G^{\text{geom}}$. Moreover, the inertia subgroup corresponding to each point in $C$ is pro-cyclic, and one may hope to give an explicit description of the action of its generator on $T_\infty(t)$. Note that the group $G^{\text{geom}}$ does not change under extension of $L$, and thus when $K$ is a number field we may calculate $G^{\text{geom}}$ over $\C$. In this case, $G^{\text{geom}}$ is given by the closure of the image of the topological fundamental group $\pi_1(\mathbb{P}^1(\C) \setminus C)$ in $\Aut(T_\infty(t))$; this image is known as the \textit{iterated monodromy group} of $\phi$. (We ignore the base point of the fundamental group, as it only affects the resulting subgroup of $\Aut(T_\infty(t))$ by a conjugation.) Then the action of inertial generators may be calculated explicitly using $\phi$-lifts of certain loops in $\C$, and one obtains a beautiful description of these generators in terms of a finite automaton. See for instance \cite{nek, neksurvey} for more on this theory.  When $K$ is a field of characteristic $> d$ and $\phi$ is post-critically finite, one may hope that inertial generators of the action of $G^{\text{geom}}$ on $T_\infty(t)$ may still be given by the states of a finite automaton. However, this is only known at present in the case where $\phi$ is a quadratic rational function \cite{pink1}. 

What, then, may be said about the group $G^{\text{arith}}$? Unfortunately, the extension $L$ in \eqref{sequence} remains mysterious in general, particularly in the case where $\phi$ is post-critically finite. An outstanding contribution of \cite{pink3} is the computation of $L$ when $\phi$ is a post-critically finite quadratic polynomial defined over a general field $K$. In particular, $[L : K]$ is finite when the orbit of the critical point of $\phi$ is pre-periodic and the post-critical set has at least 3 elements. Otherwise, $[L : K]$ is infinite. In either case, the extension $L/K$ is contained in the extension of $K$ generated by the primitive $(2^n)$th roots of unity for $n = 1, 2 \ldots$. It follows that $G^{\text{arith}}$ is a topologically finitely generated subgroup of $\Aut(T_\infty(t))$. When $\phi$ is a quadratic rational function that is \textit{not} post-critically finite, then $G^{\text{arith}}$ is completely determined in \cite{pink2}; see the discussion following Question \ref{restrictions}.  In Section \ref{overlap} we discuss $G^{\text{arith}}$ and $G^{\text{geom}}$ when $\phi$ is a non-post-critically finite map of the form $x^d + b$.

To close this subsection, we mention that it is a very interesting question, both when $\phi$ is post-critically finite and in general, to determine whether there are special properties of the conjugacy classes in $G_\infty(\phi)$ of Frobenius elements at the various primes of $K$. In the general case, the author and N. Boston have made some conjectures; we refer the reader to \cite{settled} for details. When $G_\infty(\phi)$ is a small subgroup of $\Aut(T)$, the possibility arises that special properties of the Frobenius conjugacy classes could be related to the structure of $G_\infty(\phi)$. To state this question more precisely, we note that the \textit{Hausdorff dimension} of $G_\infty(\phi)$ is by definition
$$ \liminf_{n \to \infty} \frac{\log \#G_n(\phi)}{\log \# \Aut(T_n)}.$$
\begin{question}
Suppose that the Hausdorff dimension of $G_\infty(\phi)$ is $< 1$. How does the structure of $G_\infty(\phi)$ relate to properties of Frobenius conjugacy classes?
\end{question}

\subsection{Rational functions with overlapping critical orbits} \label{overlap}
Post-critically finite maps represent an extreme among non-generic critical configurations, and it is natural to ask whether less extreme configurations also lead to restrictions on $G_\infty(\phi)$. A first remark is that even the seemingly severe restriction that $\phi$ be a polynomial, i.e. have a totally ramified fixed critical point, does not a priori impose restrictions on $G_\infty(\phi)$, as evidenced by Theorem \ref{odonigen}. Similarly, the quadratic map in \eqref{special} has one wandering critical point and one in a 2-cycle, yet has $G_\infty(\phi) = \Aut(T_\infty)$. 

On the other hand, let us consider maps of the form $\phi(x) = x^d + b$, where $d \geq 2$ and $b \in K$ is such that $0$ has infinite forward orbit under $\phi$ (or equivalently, $\phi$ is not post-critically finite). The fact that $\phi$ has only a single critical orbit besides the fixed point at infinity is enough to force $G_\infty(\phi)$ to be a very small subgroup of $\Aut(T_\infty)$. Indeed, the extension $K_{n+1}/K_n$ is obtained by adjoining the $d$th roots of $d^{n}$ elements, and hence for $n \geq 1$ has degree at most $d^{d^{n}}$ (since $K_n$ contains a primitive $d$th root of unity when $n \geq 1$). But the kernel of the restriction $\Aut(T_{n+1}) \to \Aut(T_n)$ is isomorphic to $(S_d)^{d^n}$, and thus has order $(d!)^{d^n}$. It follows that the Hausdorff dimension of $G_\infty(\phi)$ is at most $(\log d)/ (\log(d!))$, and in particular, $[\Aut(T_\infty) : G_\infty(\phi)]$ is infinite for $d \geq 3$. It follows from Stirling's formula that $(\log d)/ (\log(d!))$ is roughly $1/d$. More precisely, 
$$\frac{\log d}{\log d!} =  \left(d - \frac{d}{\ln d} + O(1) \right)^{-1}.$$  For a more thorough examination of the nature of $G_\infty(\phi)$ in this case, see \cite{zdc}. We remark that it is reasonable to expect that the image of $G^{\text{geom}}$ in $\Aut(T_n(t))$ is isomorphic to the $n$-fold wreath product of $\Z/d\Z$, for each $n \geq 1$. 
In this case, the extension $L$ in \eqref{sequence} is simply $K(\zeta_d)$, and hence $G^{\text{arith}}/G^{\text{geom}}$ has order at most $d-1$. 

In light of the preceding analysis, it seems likely that $[\Aut(T_\infty) : G_\infty(\phi)]$ is infinite whenever $\phi$ has degree at least 3 and only a single wandering critical orbit. More generally, we pose this question:

\begin{question} \label{restrictions}
Suppose that $\phi$ is not post-critically finite. What restrictions on the critical orbits of $\phi$ ensure that $[\Aut(T_\infty) : G_\infty(\phi)]$ is infinite?
\end{question}

In the case where $\phi$ is quadratic, Question \ref{restrictions} has been resolved by R. Pink \cite{pink3} as follows. If $\gamma_1$ and $\gamma_2$ are the two critical points of $\phi$, and there is a relation of the form 
\begin{equation} \label{relation}
\text{$\phi^{r+1}(\gamma_1) = \phi^{r+1}(\gamma_2)$ for some $r \geq 1$}, 
\end{equation}
then $G^{\text{arith}}$ has Hausdorff dimension $1 - 2^{-r}$ in $\Aut(T_\infty(t))$, and in particular $[\Aut(T_\infty) : G_\infty(\phi)]$ is infinite. Moreover, $G^{\text{arith}}/G^{\text{geom}}$ has order 1 or 2. In the absence of a relation of the form given in \eqref{relation}, Theorem 4.8.1(a) of \cite{pink3} gives 
\begin{equation} \label{full}
G^{\text{arith}} = G^{\text{geom}} = \Aut(T_\infty(t)).
\end{equation}
 
Let us give an example of this kind of behavior, which can be found in \cite[Example 4.9.5]{pink3}. Consider the map 
\begin{equation*}
\phi(x) = \frac{x^2 - a}{x^2 + a},
\end{equation*} 
where $a \in \Q \setminus \{0, \pm 1\}$. The critical points of $\phi$ are $0$ and $\infty$, and we have $\phi(0) = -1$, $\phi(\infty) = 1$, and $\phi^2(0) = \phi^2(\infty) = (1-a)/(1+a)$. Moreover, one checks that the stipulation that $a \not\in \{0, \pm 1\}$ implies that $\phi$ is not conjugate to any of the maps in the list of Manes-Yap \cite{manes-yap}, and thus is not post-critically finite. Hence $G^{\text{arith}}$ is a subgroup of $\Aut(T_\infty(t))$ of Hausdorff dimension $1/2$, and so $G_\infty(\phi)$ has Hausdorff dimension at most 1/2.  

 

\subsection{Rational functions for which $0$ is periodic} \label{pdc} 

Let $K$ be a global field, and recall our running assumption that for each $n \geq 1$, the solutions to $\phi^n(x) = 0$ are are distinct. Suppose that $\phi^k(0) = 0$ for some $k \geq 1$, so that $\phi$ has the cycle $0 \mapsto a_1 \mapsto a_2 \cdots \mapsto a_{k-1} \mapsto 0$ in $\mathbb{P}^1(K)$.  If we set $a_0 = 0$, then for each $n \geq 1$ we have $a_{r_n} \in \phi^{-n}(0)$, where 
$n \equiv r_n \bmod{k}$ and $0 \leq r_n \leq k-1$. 
The $a_i$ all lie in $K$, and hence each set $\phi^{-n}(0)$ contains an element of $K$, which must be fixed by all elements of $G_n(\phi)$. As $\Aut(T_n)$ acts naturally on the set $\phi^{-n}(0)$, we obtain an injection 
\begin{equation} \label{stabinj}
G_n(\phi) \hookrightarrow \Stab(a_{r_n}),
\end{equation}
where $\Stab(a_{r_n})$ denotes the stabilizer in $\Aut(T_n)$ of $a_{r_n} \in \phi^{-n}(0)$.
Now it's easy to see that $\Aut(T_n)$ acts transitively on $\phi^{-n}(0)$, and hence the orbit of $a_{r_n}$ has size $d^n$. Thus $[\Aut(T_n) : \Stab(a_{r_n})] = d^n$ by the orbit-stabilizer theorem. Therefore from \eqref{stabinj} we have $[\Aut(T_n) : G_n(\phi)] \geq d^n$, and hence $[\Aut(T_\infty) : G_\infty(\phi)]$ is infinite. Another way to say this is to note that since $\phi(a_{r_{n+1}}) = a_{r_n}$, restriction gives a natural surjection $\Stab(a_{r_{n+1}}) \to \Stab(a_{r_n})$, and we may thus define $\Stab_\infty$ to be the inverse limit of these stabilizers. Intuitively, $\Stab_\infty$ is the stabilizer in $\Aut(T_\infty)$ of a single infinite branch in $T_\infty$. Note that $[\Aut(T_\infty) : \Stab_\infty] = \infty$, since $[\Aut(T_n) : \Stab(a_{r_n})] = d^n$. Moreover, \eqref{stabinj} gives
\begin{equation} \label{stabinj2}
G_\infty(\phi) \hookrightarrow \Stab_\infty.
\end{equation}
In light of \eqref{stabinj2}, we think of $G_\infty(\phi)$ as a subgroup of $\Stab_\infty$. If there are no other special circumstances forcing $G_\infty(\phi)$ to be smaller than the generic case (such as $\phi$ being post-critically finite), then it is reasonable to expect that $[\Stab_\infty : G_\infty(\phi)]$ is finite. 
\begin{question}
Let $K$ be a global field and let $\phi \in K(x)$ satisfy $\phi^k(0) = 0$ for some $k \geq 1$. Under what conditions is it possible to prove that $[\Stab_\infty : G_\infty(\phi)]$ is finite?
\end{question}
At present, there is not a single known example of a rational function for which $[\Stab_\infty : G_\infty(\phi)]$ is finite. 
 
We close this subsection by noting that if $\phi^k(0) = 0$, then the the number of irreducible factors of the numerator of $\phi^n(x)$ is without bound as $n$ grows. Indeed, one may assume inductively that the numerator of $\phi^{k(n-1)}(x)$ has at least $n-1$ irreducible factors. But then the fact that $x$ divides the numerator of $\phi^k(x)$ implies that the numerator of $\phi^{k(n-1)}(x)$ divides the numerator of $\phi^{kn}(x)$, proving that the latter has at least $n$ irreducible factors. We return to this topic in Section \ref{stability}. \label{noevstab}




\subsection{Rational functions that commute with non-trivial M\"obius transformations} \label{autom}

Suppose that $m \in \PGL_2(K)$ satisfies
\begin{equation} \label{conditions} 
m^{-1} \circ \phi \circ m = \phi \qquad \text{and} \qquad m(0) = 0.
\end{equation}
Then $m$ acts on $T_\infty$ since $m(0) = 0$, and the action of $G_\infty(\phi)$ on $T_\infty$ commutes with that of $m$, since $m$ is defined over $K$. This is analogous to the Galois action on the Tate module of an elliptic curve commuting with the action of an endomorphism of the curve. Let $A(\phi) \leq \Aut(T_\infty)$ be the subgroup generated by the actions of all $m \in \PGL_2(K)$ satisfying the conditions in \eqref{conditions}. Then we obtain an injection
\begin{equation} \label{centralizer}
G_\infty(\phi) \hookrightarrow C(\phi),
\end{equation}
where $C(\phi)$ is the centralizer of $A(\phi)$ in $\Aut(T_\infty)$.
While $A(\phi)$ must be finite \cite{silvermandef}, and indeed its group structure is limited by the very few finite subgroups of $\PGL_2(K)$, little is known about $C(\phi)$. In particular:
\begin{conjecture}
We have $[\Aut(T_\infty) : C(\phi)] = \infty$ when $A(\phi)$ is non-trivial.
\end{conjecture}
The requirement that $A(\phi)$ be non-trivial is akin to considering an elliptic curve with complex multiplication, though in the latter setting $C(\phi)$ is a Cartan subgroup, which is both a very small subgroup of $\GL(2, \Z_\ell)$, and has the striking property of being nearly abelian. In the dynamical setting, it seems unlikely that $C(\phi)$ is close to abelian, and we will see an example of this in a moment.
A seemingly much more difficult issue than studying $C(\phi)$ is to resolve the following:
\begin{question} \label{nontrivaut}
Let $K$ be a global field and let $\phi \in K(x)$ satisfy $\#A(\phi) > 1$. Under what conditions is it possible to prove that $[C(\phi) : G_\infty(\phi)]$ is finite?
\end{question}

The only case where these issues have been studied in detail is when $\phi$ has degree 2 \cite{galrat}. Let us consider the family
\begin{equation} \label{family}
\phi(x) = \frac{b(x^2+1)}{x} \qquad (b \in K).
\end{equation}
Here $A(\phi)$ is generated by the action of the map $x \to -x$, unless $b = \pm 1/2$, but in this latter case $\phi$ is post-critically finite and so fits under the rubric of Section \ref{pcf}. The group $C(\phi)$ \label{cphiref} is studied in \cite[Section 4]{galrat}, where it is shown that $C(\phi)$ has Hausdorff dimension $1/2$, and hence $[\Aut(T_\infty) : C(\phi)] = \infty$. In spite of this, $C(\phi)$ has an index-two subgroup that is isomorphic to $\Aut(T_\infty)$ \cite[Proposition 4.1]{galrat}, a state of affairs that is made possible by the self-similarity of the tree $T_\infty$. 


Several of the main results of \cite{galrat} relate to Question \ref{nontrivaut}. For simplicity, we state them in the case $K = \Q$. 
\begin{theorem}[\cite{galrat}] \label{gallstone} Let $K = \Q$.  There is a density 0 set of primes $S \subset \Z$ such that if $b \in \Z$ is not divisible by any $p \in S$ and $\phi(x) = \frac{b(x^2+1)}{x}$, then $G_\infty(\phi) \cong C(\phi)$.
\end{theorem}
In fact the set $S$ is given explicitly: it is the set of primes dividing the numerator of $\phi_1^n(1)$ for some $n \geq 1$, where $\phi_1 = (x^2+1)/x$.  All $p \in S$ satisfy $p \equiv 1 \pmod{4}$.  
In particular, the theorem applies to $$b = 1, 3, 7, 9, 11, 13, 17, 19, 21, 23, 27, 31, 33, 37, 39, 43, 47, 49, \ldots$$ It would be interesting to obtain a similar result with weaker hypotheses on $b$, which may well be possible by refining the methods of \cite{galrat}. Another consequence of the work in \cite{galrat} is:
\begin{theorem}[\cite{galrat}]
Let assumptions and notation be as in Theorem \ref{gallstone}. Then we have $[C(\phi) : G_\infty(\phi)] < \infty$ for $b \equiv 2, 3 \pmod{5}$ and $b \equiv 1, 2, 5, 6 \pmod{7}$. In addition $[C(\phi) : G_\infty(\phi)] < \infty$ for all $b \in \Z$ with $1 \leq |b| \leq 10,000$.
\end{theorem}

The proofs of these two results follow lines similar to the proof of Theorem \ref{quadfin} (see Theorem 5.3 of \cite{galrat} and the remark following for an analogue of Theorem \ref{quadfin}). On the one hand, the argument requires developing considerable machinery to handle the fact that $\phi$ is a rational function rather than a polynomial, but on the other hand it is easier in that $0$ has an extremely simple orbit under $\phi$, being sent directly to the fixed point $\infty$. Another key to the proof is that there is essentially only one critical orbit whose arithmetic one must keep track of: while technically there are two, one is the image of the other under $x \mapsto -x$. 
As with the maps in Theorem \ref{quadfin}, one finds that $[C(\phi) : G_\infty(\phi)] < \infty$ follows from the seemingly much weaker assertion that the numerators of $\phi^n(x)$ are irreducible for all $n \geq 1$. 

In light of the analysis in \cite{galrat}, we make the following conjecture:
\begin{conjecture}[\cite{galrat}] \label{famconj}
Let $K = \Q$.  If $\phi(x) = \frac{b(x^2+1)}{x}$ with $b \in \Q$ and $b \not\in \{0, \pm \frac{1}{2}\}$, then $[C(\phi) : G_\infty(\phi)] < \infty$.
\end{conjecture}
Our restriction to the family in \eqref{family} is not as significant as it may seem, as every degree 2 rational function that commutes with a non-trivial M\"obius function is conjugate to one of the form \eqref{family} (see \cite[Section 2]{galrat}). Indeed, Conjecture \ref{famconj} is equivalent to the $K = \Q$ case of the following conjecture:
\begin{conjecture}[\cite{galrat}]
Let $K$ be a global field of characteristic $0$ or $> 2$, and suppose $\phi(x) \in K(x)$ has degree 2.  Assume that $\phi$ is not post-critically finite and $0$ is not periodic under $\phi$.  If $\phi$ commutes with a non-trivial M\"obius transformation that fixes $0$, then $[C(\phi) : G_\infty(\phi)] < \infty$.
\end{conjecture}

\subsection{A conjecture for quadratic rational functions}

In light of the results on quadratic rational functions given in the previous four subsections we pose the following conjecture:
\begin{conjecture} \label{quadratic}
Let $K$ be a global field and suppose that $\phi \in K(x)$ has degree two. Then $[\Aut(T_\infty) : G_\infty(\phi)] = \infty$ if and only if one of the following holds:
\begin{enumerate}
\item The map $\phi$ is post-critically finite.
\item The two critical points $\gamma_1$ and $\gamma_2$ of $\phi$ have a relation of the form $\phi^{r+1}(\gamma_1) = \phi^{r+1}(\gamma_2)$ for some $r \geq 1$.
\item The root $0$ of $T_\infty$ is periodic under $\phi$.
\item There is a non-trivial M\"obius transformation that commutes with $\phi$ and fixes $0$. 
\end{enumerate}
\end{conjecture}
Our rationale for this conjecture is as follows. Thanks to the result of \cite{pink3} given in \eqref{full}, any quadratic rational map not satisfying condition (1) or (2) of the conjecture must satisfy $G^{\text{arith}} = \Aut(T_\infty(t))$. Hence these are the only quadratic maps for which there may be a geometric reason that $[\Aut(T_\infty) : G_\infty(\phi)] = \infty$. Among quadratic maps with $G^{\text{arith}} = \Aut(T_\infty(t))$, the only known examples where $[\Aut(T_\infty) : G_\infty(\phi)] = \infty$ are those satisfying conditions (3) and (4). The meat of the conjecture is that these are all such examples. 

We remark that if $\phi$ satisfies one of the four conditions of Conjecture \ref{quadratic}, then indeed $[\Aut(T_\infty) : G_\infty(\phi)] = \infty$. This is thanks to Theorem \ref{infindex}, results of R. Pink \cite[Theorem 4.8.1(b) and Corollary 4.8.9]{pink3}, and the fact that $[\Aut(T_\infty) : \Stab_\infty] = [\Aut(T_\infty) : C(\phi)] = \infty$, where $\Stab_\infty$ is defined in Section \ref{overlap} and $C(\phi)$ is the centralizer in $\Aut(T_\infty)$ of the action of $x \mapsto -x$ on $T_\infty$. 
The ``only if" part of Conjecture \ref{quadratic} remains wide open. 

\section{Density results} \label{density}

Let us return now to the study of the density of prime divisors of orbits of rational functions, which motivated the initial investigations into arboreal representations. 
We show in this section that, happily, one may obtain zero-density results with a significantly weaker hypothesis than $G_\infty(\phi)$ having finite index in a known subgroup of $\Aut(T_\infty)$. 

For a general global field $K$, we have two notions of density available for a set $S$ of primes of $K$: \label{densedef}
$$ \lim_{s \rightarrow 1^{+}} \frac{\sum_{{\mathfrak{q}} \in S} N({\mathfrak{q}})^{-s}} {\sum_{{\mathfrak{q}}}N({\mathfrak{q}})^{-s}} \qquad \text{and} \qquad 
\limsup_{x \to \infty} \frac{\#\{\mathfrak{q} \in S : N({\mathfrak{q}}) \leq x\}}{\#\{{\mathfrak{q}} : N({\mathfrak{q}}) \leq x\}},$$
where $N({\mathfrak{q}}) = \#({\mathcal{O}}_K/{\mathfrak{q}}{\mathcal{O}}_K)$, and the sum in each denominator runs over all primes of $K$. 
The quantity on the left is called \textit{Dirichlet density}, while that on the right is \textit{natural density}. When the natural density of $S$ exists, then so does its Dirichlet density, and the two coincide. Moreover, there are sets for which the Dirichlet density exists but the natural density does not. Natural density earns its name because it corresponds more closely to the intuitive notion of the limiting probability as $x \to \infty$ that a randomly chosen prime $\leq x$ belongs to $S$. Because of the differences in the versions of the Chebotarev density theorem that hold over function fields and number fields (see \cite[p.125]{Rosen} for the former and \cite[p. 368]{narkiewicz} for the latter), we use Dirichlet density in the function field setting and natural density in the number field setting. From now on, this is what we mean by ``the density" of a set of primes. 

Let $K$ be a global field, $\phi \in K(x)$, and $a_0 \in K$. Let $v_\p$ denote the $\p$-adic valuation for a prime $\p$ of $K$, and define 
$$P_\phi(a_0) := \{\text{$\p$} : \text{$v_\p(\phi^i(a_0)) > 0$ for at least one $i \geq 0$ with $\phi^i(a_0) \neq 0$}\}.$$
We denote $v_\p(\phi^i(a_0)) > 0$ by $\p \mid \phi^i(a_0)$.  As noted in the discussion on p. \pageref{frobdisc}, when $\phi(x)$ is a polynomial, the density of the complement of $P_\phi(a_0)$ is bounded below by the density of $\p$ such that $\phi^n(x) \equiv 0 \bmod{\p}$ has no solution. \label{bound} For if $\p$ satisfies this condition, then we cannot have $\p \mid \phi^j(a_0)$ for $j \geq n$, since otherwise 
$\phi^n(x) \equiv 0 \bmod{\p}$ has a solution with $x = \phi^{j-n}(a_0)$. However, only finitely many $\p$ satisfy $\p \mid \phi^j(a_0)$ for $0 \leq j < n$.  A similar conclusion holds when $\phi$ is a rational function, but one must require $\phi^n(\infty) \neq 0$ and discard the finitely many $\p$ dividing $\phi^n(\infty)$ and where $\phi$ has bad reduction. See \cite[Theorem 6.1]{galrat} for details. 

Now $\phi^n(x) \equiv 0 \bmod{\p}$ having no solution is equivalent to Frobenius at $\p$ acting without fixed points on the elements of $\phi^{-n}(0)$. One then gets from the Chebotarev density theorem that the density  of $P_\phi(a_0)$ is bounded above by the proportion of elements of 
$G_n(\phi)$ that act on $\phi^{-n}(0)$ with at least one fixed point. This holds for any $n$, and it follows that $P_\phi(a_0)$ has density zero provided that
\begin{equation} \label{thekey}
\lim_{n \to \infty} \frac{\# \{g \in G_n(\phi) : \text{$g$ fixes at least one element of $\phi^{-n}(0)$} \}}{\#G_n(\phi)} = 0.
\end{equation}
The sequence in the limit is non-increasing, for if $g \in G_n(\phi)$ acts on $\phi^{-n}(0)$ without fixed points, then the same is true of all $g' \in G_{n+1}(\phi)$ that restrict to $g$. Therefore the limit in \eqref{thekey} exists. 

In the relatively rare cases where $G_n(\phi)$ is known explicitly for all $n \geq 1$, the limit in \eqref{thekey} can be calculated directly. The following result combines \cite[Theorem 6.2]{galrat} and \cite[Propositions 5.5, 5.6]{galmart}.
\begin{theorem}
Let $K$ be a global field, let $\phi(x) \in K(x)$, and let $C(\phi)$ be defined as in the discussion following \eqref{family}. If $G_\infty(\phi) = \Aut(T_\infty)$ or $G_\infty(\phi) = C(\phi)$, then the density of $P_\phi(a_0)$ is zero for all $a_0 \in K$. 
\end{theorem}
This establishes zero-density results for orbits of the $\phi$ given in Theorems \ref{stollthm} and \ref{gallstone}, as well as the map in \eqref{special}. A better result will supersede this, however, once we introduce some new ideas that allow for a similar conclusion with vastly less knowledge of $G_\infty(\phi)$. In early 2004, the author was able to establish an important fact about $G_\infty(\phi)$ (see \eqref{essential}) in the setting $K = \F_p(t)$ ($p$ an odd prime) and $\phi(x) = x^2 + t$, but saw no way to translate this into a form that would help prove \eqref{thekey}. However, in a fortuitous conversation after a basketball game, A. Hoffman (then an applied math graduate student at Brown University) suggested that a convergence theorem from probability theory might be just the ticket. The resulting change in viewpoint led to the main theorems of \cite{galmart}, and, not coincidentally, the author's successful completion of graduate school. 

In light of \eqref{thekey}, we wish to measure the probability of a randomly chosen element of $G_n(\phi)$ belonging to the set given in the numerator of the expression in \eqref{thekey}, and more precisely how this probability evolves as $n$ grows. It's useful therefore to associate to a given $g \in G_n(\phi)$ the sequence $X_1(g), X_2(g), \ldots, X_n(g)$, where $X_i(g)$ is the number of elements of $\phi^{-i}(0)$ fixed by $g$ (recall that $g$ acts on $\phi^{-i}(0)$ for $1 \leq i \leq n$ through the restriction map $G_n(\phi) \to G_i(\phi)$). If the limit in \eqref{thekey} is zero, then when $n$ is large almost any choice of $g$ will result in a sequence that has reached zero by the $n$th term. To understand the actual limit as $n$ tends to infinity, we should work in $G_\infty(\phi)$, and use the restriction maps $\pi_i : G_\infty(\phi) \to G_i(\phi)$. Happily, $G_\infty(\phi)$ has a natural probability measure $\P$ given by the normalized Haar measure, with the excellent property that for any $S \subseteq G_i(\phi)$, 
$\P(\pi_i^{-1}(S)) = \#S/\#G_i(\phi)$. We can now translate \eqref{thekey} into 
\begin{equation} \label{probkey}
\lim_{n \to \infty} \P(g \in G_\infty(\phi) : X_n(g) > 0) = 0.
\end{equation} 
To each $g \in G_\infty(\phi)$, we attach the infinite sequence $X_1(g), X_2(g), \ldots$. Note that the $X_i$ are random variables on the probability space $G_\infty(\phi)$, and probabilists are wont to give any infinite sequence of random variables on a fixed probability space the fancy-sounding moniker \textit{stochastic process}. As this process $X_1, X_2, \ldots$ encodes information about the Galois action on $T_\infty$, we call it the \textit{Galois process} of $\phi$. 

This rephrasing of our group theory problem in probabilistic terms has value in that it allows us to use the considerable machinery of the theory of stochastic processes. Because $\lim_{n \to \infty} \P(S_n) = \P(\bigcap S_n)$ for any nested sequence of sets $S_n$, \eqref{probkey} is equivalent to the statement that almost all sequences $X_1(g), X_2(g), \ldots$ are eventually zero. 
To prove this, we use two steps:
\begin{enumerate} \label{twosteps}
\item[(A)] Show that almost all sequences $X_1(g), X_2(g), \ldots$ are eventually constant. 
\item[(B)] Show that if $r > 0$, then for infinitely many $n \geq 1$ we have $$\P(X_{n}(g) = r \mid X_{n-1}(g) = r) \leq 1 - \epsilon,$$ where $\epsilon > 0$ is independent of $n$. 
\end{enumerate}
Condition (B) ensures that the probability of $X_1(g), X_2(g), \ldots$ being eventually constant at a fixed $r > 0$ is zero, and the desired conclusion follows. While step (A) may not seem the most obvious way to proceed, it fits nicely with the notion of convergence of a stochastic process: the process $X_1, X_2, \ldots$ converges if there exists a random variable $X : G_\infty \to \mathbb{R}$ such that $X_n \to X$ almost surely, or in other words,
$$\P(g \in G_\infty(\phi) : \text{$\lim_{n \to \infty} X_n(g)$ exists}) = 1.$$
Because the $X_n$ are integer-valued, this implies that the sequence $X_1(g), X_2(g), \ldots$ is eventually constant with probability one, just as in (A) above. 

But how to show the Galois process converges? It is here that we call on substantial ideas from probability theory, which has a plethora of results giving sufficient conditions for a stochastic process to converge. One kind of process for which powerful convergence theorems exist is called a \textit{martingale}, which roughly is a ``locally fair" process in that the expected behavior one step into the future, given a certain present behavior, is always the same as the present behavior. More precisely, for all $n \geq 2$ and any $t_i \in \R$, 
\begin{equation} \label{mart}
E(X_n \mid X_{1} = t_{1}, X_2 = t_2, \ldots, X_{n-1} = t_{n-1}) = t_{n-1},
\end{equation}
provided $\P(X_{1} = t_{1}, X_2 = t_2, \ldots, X_{n-1} = t_{n-1}) > 0$. Martingales often converge; in particular \cite[Section 12.3]{grimmett} gives the highly useful result that if the random variables of a martingale take non-negative values, then the martingale converges. Certainly in the present case $X_n(g) \geq 0$ for all $g \in G_\infty(\phi)$. To sum up, then, we may accomplish step (A) above simply by showing that the Galois process is a martingale. 

To establish \eqref{mart} for the Galois process involves examining all lifts to $G_{n}(\phi)$ of a given $g_0 \in G_{n-1}(\phi)$. Indeed, conditioning on the behavior $X_{1} = t_{1}, X_2 = t_2, \ldots, X_{n-1} = t_{n-1}$ is the same as restricting consideration to a certain subset $S$ of $G_{n-1}(\phi)$, and then looking at the expected value of $X_n(g)$ as $g \in G_{n}(\phi)$ varies over elements restricting to $S$. If we can show that the expected value of $X_n(g)$ is $t_{n-1}$ for lifts of each $g_0 \in S$ individually, then \eqref{mart} immediately follows. 

Now the set of all lifts to $G_{n}(\phi)$ of $g_0 \in G_{n-1}(\phi)$ is just the coset $gH_n$, where $g$ is any lift of $g_0$ and 
$$H_n = \{h \in G_n(\phi) : \text{$h$ restricts to the identity on $G_{n-1}(\phi)$}\}.$$ 
If we let $K_n = K(\phi^{-n}(0))$, then $H_n$ is the Galois group of the relative extension $K_n/K_{n-1}$. Because we are conditioning on 
$X_{1} = t_{1}, X_2 = t_2, \ldots, X_{n-1} = t_{n-1}$, we may assume that $g_0$ has $t_{n-1}$ fixed points in $\phi^{-(n-1)}(0)$. Let $\alpha$ be one such fixed point, and note that to establish \eqref{mart} it is enough to show that on average an element of $gH_n$ fixes one point in $\phi^{-1}(\alpha)$, for then the average total number of fixed points of an element of $gH_n$ acting on $\phi^{-n}(0)$ is $t_{n-1}$. Now if 
\begin{equation} \label{essential}
\text{$H_n$ acts transitively on every set of the form $\phi^{-1}(\alpha)$}, 
\end{equation}
then an application of Burnside's lemma gives the desired result. In the case where $\phi$ is a polynomial, \eqref{essential} is equivalent to $\phi(x) - \alpha$ being irreducible over $K_{n-1}$ for each root $\alpha$ of $\phi^{n-1}(x)$. See \cite[Theorem 2.5]{quaddiv} for a slightly more general version of this statement. 

Establishing \eqref{essential} is difficult in general, but turns out to be tractable in many circumstances. In the geometric setting considered in \cite{fpfree}, it is an easy result (see the remarks following Theorem 5.1 of \cite{fpfree}). 
The first result in an arithmetic setting appeared in \cite[Theorem 1.2]{galmart}, in the case where $\phi$ is a quadratic polynomial over a field of characteristic $\neq 2$ satisfying a hypothesis that essentially says the critical orbit of $\phi$ contains no squares. A mild generalization, allowing roughly for a finite number of squares to occur in the critical orbit of $\phi$, appeared in \cite[Theorem 2.7]{quaddiv}. In particular this result implies that if $\phi$ is quadratic with all iterates irreducible, then \eqref{essential} holds for sufficiently large $n$, and this is enough to establish (A). A more significant generalization to certain polynomials of the form $x^p + b$, where $p$ is prime, has recently been given in \cite[Theorem 3.4]{zdc}, under the hypothesis that the ground field $K$ contains a primitive $p$th root of unity. The proofs of all these theorems rely on a careful study of permutation groups with certain properties. A different approach is taken in \cite[Theorem 3.3]{zdc}, where a much more straightforward local argument suffices to prove \eqref{essential} for $\phi(x) = x^d + b$ under the slightly more restrictive hypothesis that $v_\p(b) > 0$ for some prime $\p$ of $K$ of residue characteristic not dividing $d$, but with the great added advantage of holding for composite $d$. Again, $K$ is assumed to contain a primitive $d$th root of unity. 

We turn now to proving (B), the second step in the two-step program given on p. \pageref{twosteps}. As in step (A), the probability involved is conditioned on the value of $X_{n-1}(g)$, and thus we may restrict consideration to cosets of $H_n$. The key advantage is that \textit{we only need knowledge of $H_n$ for infinitely many $n$}. In many cases it is possible to precisely determine $H_n$ for an infinite set of $n$. This is certainly true when $[\Aut(T_\infty) : G_\infty(\phi)]$ is finite, and while the literature appears to contain no precise proof of (B) in this case, one can be adapted from \cite[Lemma 4.6]{zdc} (see also \cite[Lemma 4.3]{odonigalit}). This gives
\begin{theorem}
Let $K$ be a global field and $\phi \in K(x)$. Suppose that the Galois process for $\phi$ is a martingale and $[\Aut(T_\infty) : G_\infty(\phi)]$ is finite. Then the density of $P_\phi(a_0)$ is zero for all $a_0 \in K$.
\end{theorem}

In many cases, zero-density results are possible under far weaker assumptions than $[\Aut(T_\infty) : G_\infty(\phi)] < \infty$. For instance, when $\phi$ is a quadratic polynomial, one can show under mild hypotheses that Siegel's theorem on integral points implies $H_n \cong (\Z/2\Z)^{2^{n-1}}$ (that is, $H_n$ is as large as possible) for infinitely many $n$. See \cite[Corollary 6.6]{galmart} and \cite[Proof of Theorem 1.1]{quaddiv}. In particular, one may obtain a zero-density result for primes dividing orbits of $\phi(x) = x^2 + 3$, though as mentioned on p. \pageref{3case} it is not known whether $[\Aut(T_{\infty}): G_\infty(\phi)] < \infty$. Another interesting example is that of $\phi(x) = x^2 + t$, with $K = \Fp(t)$ for an odd prime $p$; this is the motivating example of \cite{galmart}. In this case, one can show that $H_n \cong (\Z/2\Z)^{2^{n-1}}$ when $n$ is squarefree \cite[Corollary 6.6]{galmart}, although it remains unknown whether $[\Aut(T_\infty) : G_\infty(\phi)] < \infty$ (see Conjecture 6.7 of \cite{galmart}). This yields a zero-density result for prime divisors of orbits of $\phi$, and in particular for prime divisors of the sequence 
$$ \{t, t^2 + t, t^4 + 2t^3 + t^2 + t, \ldots\},$$
which is the orbit of $0$ under $\phi$ in $\Fp[t]$. This has consequences for the $p$-adic Mandelbrot set, in particular showing that its hyperbolic subset is small in a certain sense (see Theorem 1.4 of \cite{galmart}). 
A further interesting family of examples is given by $\phi(x) = x^d + b$.  For this family, it's shown in \cite[Theorem 4.5]{zdc} that $H_n \cong (\Z/d\Z)^{d^{n-1}}$ for infinitely many $n$, under mild conditions on $b$. This leads to corresponding zero-density results (see Theorem 1.1 of \cite{zdc}, or part (5) of Theorem \ref{sumup} below). 
A crucial caveat in all the results mentioned in this paragraph is that they require that all iterates of $\phi$ be irreducible, pointing up once again the importance of this property. 

As a final note, many of the polynomial results cited in this section are proven for \textit{translated iterates}, that is, polynomials of the form 
$g \circ \phi^n(x)$, where $g(x)$ divides some iterate of $\phi$. This presents only mild complications and allows one to obtain density results 
in the situation where some iterates of $\phi(x)$ are reducible, provided that the number of irreducible factors of $\phi^n(x)$ is bounded as $n$ grows (in the terminology of Section \ref{stability}, $\phi$ is \textit{eventually stable}). For example, this makes possible density results about $\phi(x) = x^2 - 4$, which has the property that for each $n \geq 1$, $\phi^n(x)$ is the product of two irreducible polynomials over $\Q$ (see \cite[Section 4]{quaddiv}). 

We now give a theorem that exemplifies the kind of result made possible by the preceding analysis. Each statement below is a special case of the theorem cited. 



\begin{theorem} \label{sumup} For the following $\phi \in \Q(x)$, $P_\phi(a_0)$ has density zero for all $a_0 \in \Q$:
\begin{enumerate}
\item $\phi(x) = x^2 + kx - k$ for $k\in \Z$ \; \cite[Theorem 1.2]{quaddiv}
\item $\phi(x) = x^2 + kx - 1$ for $k \in \Z \setminus \{0,2\}$ \; \cite[Theorem 1.2]{quaddiv}
\item $\phi(x) = x^2 + k$ for $k \in \Z \setminus \{- 1\}$ \; \cite[Theorem 1.2]{quaddiv}
\item $\phi(x) = \frac{k(x^2 + 1)}{x}$ for odd $k \in \Z$ having no prime factor $\equiv 1 \bmod{4}$ \; \cite[Corollary 5.14, Theorem 6.2]{galrat}
\end{enumerate}
Moreover, if $p$ is an odd prime, $K$ is a number field containing a primitive $p$th root of unity, and 

\begin{enumerate}
\item[(5)] $\phi(x) = x^p + k$ for $k \in \Z$,
\end{enumerate}
then $P_\phi(a_0)$ has density zero for all $a_0 \in K$ \cite[Corollary 1.3]{zdc}.
\end{theorem}





\section{Stability and eventual stability} \label{stability}

As noted frequently in Sections \ref{generic} and \ref{density}, establishing the transitivity of the action of $G_n(\phi)$ on the sets $\phi^{-n}(0)$ is crucial to understanding $G_\infty(\phi)$. Even when this transitivity fails, one can often recover significant information about $G_\infty(\phi)$ when its action on $\phi^{-n}(0)$ has a bounded number of orbits as $n$ grows. Thus we are interested in the factorization into irreducibles of the numerator of $\phi^n(x)$. We make these definitions, where $F$ denotes \textit{any field}:

\smallskip

\begin{itemize}
\item $\phi \in F(x)$ is \textit{stable} if the numerator of $\phi^n(x)$ is irreducible for all $n \geq 1$.
\item $\phi \in F(x)$ is \textit{eventually stable} if the number of irreducible factors of the numerator of $\phi^n(x)$ is bounded as $n$ grows. 
\end{itemize}

The extent to which these two properties hold for generic $\phi$ is a question of great interest, and which has prompted much recent research. As in the study of the Galois theory of iterates, it was Odoni who first examined questions of stability: see \cite[Sections 1 and 2]{odonigalit}, \cite[Proposition 4.1]{odoniwn}, and \cite[Lemma 4.2]{odoni}. A fundamental observation is that Eisenstein polynomials are stable, as any iterate of an Eisenstein polynomial is again Eisenstein. This statement holds in great generality, and in \cite[Lemma 2.2]{odonigalit} Odoni uses it to prove that the generic degree-$d$ monic polynomial given in  \eqref{genmon} is stable. When $\phi$ is a quadratic polynomial, recent work gives additional sufficient conditions for stability to hold. The critical point of $\phi$ again proves critical, just as in the questions of the maximality of $H_n$ dealt with in Sections \ref{generic} and \ref{density}. Here are two such results:

\begin{theorem} \cite[Theorem 2.2]{itconst} \label{stab}
Let $F$ be any field of characteristic $\neq 2$, and let $\phi \in F[x]$ be monic and quadratic, with critical point $c$.  Then $\phi(x)$ is stable if none of $-\phi(c), \phi^2(c), \phi^3(c), \phi^4(c) \ldots$ is a square in $F$.  
\end{theorem}

\begin{theorem} \cite[Theorem 3.1]{itconst} \label{stab2}
Let $\phi(x) = (x - \gamma)^2 + \gamma + m$ with $\gamma, m \in \Z$. If $\gamma \not\equiv m \bmod{2}$, then $\phi$ is stable. 
\end{theorem}

Both of these results apply to many non-Eisenstein polynomials. 
When the field $F$ in Theorem \ref{stab} is a finite field, ``if" may be replaced by ``if and only if," and this stronger result underlies much of the analysis in \cite{settled}. The proof of Theorem \ref{stab} is a nice exercise in field theory, with the key step being to define a certain sequence $(\tau_n)_{n \geq 1}$ with $\tau_n \in F(\phi^{-n}(0))$, and to show that $\tau_n$ is not a square in $F(\phi^{-n}(0))$, for each $n \geq 1$. To do this, one takes the norm from $F(\phi^{-n}(0))$ to $F$ of $\tau_n$, and the result is a square times $\phi^n(c)$. Hence if $\phi^n(c)$ is not a square in $F$ for each $n \geq 1$, the desired result follows (with an additional complication in the case $n = 1$). Theorem \ref{stab2} is proven by taking the norm of $\tau_n$ from $F(\phi^{-n}(0))$ to $F(\phi^{-1}(0))$ instead. The version stated here is a special case of \cite[Theorem 3.1]{itconst}, as the latter holds over most number fields. 

When $\phi$ is a rational function, even of degree 2, there are very few results giving sufficient conditions for $\phi$ to be stable. One such result is for the family in \eqref{family}, where a condition similar to that of Theorem \ref{stab} is given in \cite[Theorem 4.5]{galrat}.

The fact that Eisenstein polynomials are stable, along with Theorems \ref{stab} and \ref{stab2}, suggests that stability should hold for a large class of polynomials over a given global field.  Indeed, when $\phi \in \Z[x]$ is monic and quadratic this is a theorem (see \cite{shparostafe}, where a proof is given using Theorem \ref{stab}). However, the notion of stability has the disadvantage of not being invariant under finite extensions of the ground field. Moreover, even for quadratic polynomials over $\Q$ one finds examples where stability fails for no obvious structural reason. For instance, recall from p. \pageref{cooled} the case $\phi(x) = x^2 - x - 1$, where $\phi(x)$ and $\phi^2(x)$ are irreducible but $\phi^3(x)$ factors as the product of two irreducible quartics. Another interesting example is $\phi(x) = x^2-\frac{16}{9}$, where one has not only the obvious factorization of $\phi$, but an additional splitting of $\phi^3$:
\[\phi^3(x)  =  \left(x^2-2x+\frac{2}{9}\right) \left(x^2+2x+\frac{2}{9}\right) \left(x^2-\frac{22}{9}\right) \left(x^2-\frac{10}{9}\right).\]
It is possible to prove for this example that no additional splitting occurs: for $n \geq 3$, $\phi^n(x)$ has precisely four irreducible factors over ${\mathbb{Q}}$ (see the remark following the proof of Theorem 1.6 of \cite{zdc}).

Eventual stability, on the other hand, may reasonably be expected to hold for all maps for which $0$ is not periodic under $\phi$ (see the discussion at the end of Section \ref{pdc} for the reasons why the latter must be excluded). In the case where $\phi \in \Z[x]$ is monic and quadratic, this is Conjecture 1 at the end of Section 4 of \cite{quaddiv}. A more general conjecture is proposed in \cite{evstab}. However, few results in this direction are known. To the author's knowledge, the most general are these:
\begin{theorem} \cite[Theorem 1.6]{zdc} \label{evstab1}
Let $d \geq 2$, let $K$ be a field of characteristic not dividing $d$, and let $\phi(x) = x^d + c \in K[x]$ with $c \neq 0$.  If there is a discrete non-archimedean absolute value on $K$ with $|c| < 1$, then $\phi$ is eventually stable over $K$. 
\end{theorem}
\begin{theorem} \cite[Corollary 3]{ingram} \label{evstab2}
Let $K$ be a number field and $\phi(x)$ a monic polynomial of degree $d$ defined over $K$. Suppose that there exists a non-archimedean prime ${\mathfrak{p}}$ of $K$ with ${\mathfrak{p}} \nmid d$ and such that $|\phi^n(0)|_{\mathfrak{p}} \to \infty$ as $n \to \infty$. Then $\phi$ is eventually stable over $K$. 
\end{theorem}
See also \cite[Proposition 4.5]{quaddiv}, where eventual stability is proven for some specific families of quadratic polynomials over $\Z$. Theorem \ref{evstab1} gives an especially strong result in the case where $K$ is a global function field (or indeed a function field over any field) of characteristic not dividing $d$: $\phi$ is eventually stable unless $c$ belongs to the field of constants of $K$. See \cite[Corollary 1.8]{zdc}. Interestingly, the maps in Theorem \ref{evstab1} satisfy $|\phi^n(0)| \to 0$ as $n \to \infty$, and so Theorems \ref{evstab1} and \ref{evstab2} apply to quite different maps. The methods of proof of both are local in nature, but the proof of Theorem \ref{evstab1} relies on the fact that factorizations of iterates of $x^d + c$ take a special form \cite[Theorem 2.2]{zdc}, while to prove Theorem \ref{evstab2}, Ingram constructs a non-archimedean version of the B\"ottcher coordinate \cite[Theorem 2]{ingram}. 
 
Questions of stability and eventual stability remain at the heart of this area, and a subject of active research. See for instance \cite{shparostafe, HCsummer, danielson, ostafe, shparlinski, sookdeo} for further reading.  

\section*{Acknowledgements}
I am grateful to Richard Pink, Joe Silverman, Rob Benedetto, Ben Hutz, and Wade Hindes for valuable comments on earlier drafts of this article. 






\noindent

\bibliographystyle{plain}

\end{document}